\documentclass{amsart}
\usepackage{amssymb}
\usepackage{mathrsfs}
\usepackage{stmaryrd}
\usepackage{bbm}
\usepackage{oldgerm}
\usepackage[english]{babel}
\usepackage[T1]{fontenc}
\usepackage[latin1]{inputenc}
\usepackage[all]{xy}
\usepackage{mathabx}
\usepackage{graphicx}
\usepackage[colorlinks,linkcolor=red,anchorcolor=red,citecolor=blue]{hyperref}

\newtheorem{theorem}[subsection]{Theorem}
\newtheorem{proposition}[subsection]{Proposition}
\newtheorem{corollary}[subsection]{Corollary}

\theoremstyle{definition}

\newtheorem{remark}[subsection]{Remark}

\numberwithin{equation}{subsection}

\newcommand{\sF}{\mathscr{F}}
\newcommand{\ra}{\rightarrow}

\newcommand{\bA}{\mathbb A}

\begin{document}

\title{semi-continuity for total dimension divisors of \'etale sheaves}

\author{Haoyu Hu}
\address{Graduate School of Mathematical Sciences, the University of Tokyo, 3-8-1 Komaba Meguro-ku Tokyo 153-8914, Japan}
\email{huhaoyu@ms.u-tokyo.ac.jp, huhaoyu1987@gmail.com}

\author{Enlin Yang}
\address{FU Berlin,  Mathematik und Informatik, Institut f\"{u}r Mathematik, Arnimallee 6, 14195 Berlin, Germany}
\email{ enlin.yang@fu-berlin.de, yangenlin0727@126.com}



\subjclass[2000]{Primary 14F20; Secondary 11S15}



\keywords{Semi-continuity, ramification, total dimension divisors}

\begin{abstract}
In this article, we extend an inequality that compares the pull-back of the total dimension divisor of an \'etale sheaf and the total dimension divisor of the pull-back of the sheaf due to Saito. Using this formula, we generalize Deligne and Laumon's lower semi-continuity property for Swan conductors of \'etale sheaves on relative curves to higher relative dimensions in a geometric situation.
\end{abstract}
\maketitle

\tableofcontents

\section{Introduction}
\subsection{}\label{delignelaumon}
Let $S$ be an excellent noetherian scheme, $f:X\rightarrow S$ a separated and smooth morphism of relative dimension $1$, $D$ a closed subset of $X$ which is finite and flat over $S$, $U$ the complement of $D$ in $X$ and $j:U\rightarrow X$ the canonical injection. Let $\ell$ be a prime number invertible in $S$ and $\mathscr F$ a locally constant and constructible sheaf of $\mathbb F_{\ell}$-modules on $U$ of constant rank. For any point $ s\in S$, we denote by $\bar s\rightarrow S$ an algebraic geometric point above $s$ (cf. \ref{Notgeom}) and by $X_{\bar s}$ and $D_{\bar s}$ the fibers of $f:X\rightarrow S$ and $f|_D:D\rightarrow S$ at $\bar s$, respectively. For each point $x\in D_{\bar s}$, we define
\begin{equation}\label{dimtotcurve}
\mathrm{dimtot}_x(j_!\mathscr F|_{X_{\bar s}})=\mathrm{Sw}_x(j_!\mathscr F|_{X_{\bar s}})+\mathrm{rank}(\mathscr F),
\end{equation}
where $\mathrm{Sw}_x(j_!\mathscr F|_{X_{\bar s}})$ denotes the classical Swan conductor of the sheaf $j_!\mathscr F|_{X_{\bar s}}$ at $x$ which is an integer number \cite[19.3]{linrep}. The sum
\begin{equation}\label{introsumdimtot}
\sum_{x\in D_{\bar s}}\mathrm{dimtot}_x(j_!\mathscr F|_{X_{\bar s}})
\end{equation}
does not depend on the choice of $\bar s$ above $s$. It defines a function $\varphi:S\rightarrow \mathbb Z$. The following property of $\varphi$ is due to Deligne and Laumon:

\begin{theorem}[{\cite[2.1.1]{lau}}]\label{themdelignelaumon}
We take the notation and assumptions of \ref{delignelaumon}. Then,
\begin{itemize}
\item[(1)]
The function $\varphi:S\rightarrow \mathbb Z$ is constructible and lower semi-continuous on $S$.
\item[(2)]
If $\varphi:S\rightarrow \mathbb Z$ is locally constant, then $f:X\rightarrow S$ is universally locally acyclic with respect to $j_!\mathscr F$.
\end{itemize}
\end{theorem}

\subsection{}
Assume that $f:X\rightarrow S$ has relative dimension $\geq1$, and $D$ is a Cartier divisor on $X$ relative to $S$. For each algebraic geometric point $\bar s$ of $S$, the ramification of $\mathscr F|_{U_{\bar s}}$ along $D_{\bar s}$ defines a divisor supported on $D_{\bar s}$ which is called the {\it total dimension divisor} of $\mathscr F|_{U_{\bar s}}$ on $X_{\bar s}$ (cf. \eqref{introdefdimtot}). It is a generalization of the sum \eqref{introsumdimtot}. This article is devoted to proving the semi-continuity of the total dimension divisor of \'etale sheaves in a geometric setting, that generalizes the property (1) of theorem \ref{themdelignelaumon} (cf. theorem \ref{special statement} and theorem \ref{thscdiv}).

\subsection{}
Let $K$ be a complete discrete valuation field, $\mathscr O_K$ its integer ring and $F$ the residue field of $\mathscr O_K$. We assume that the characteristic of $F$ is $p>0$. Let $\overline K$ be a separable closure of $K$ and we denote by $G_K$ the Galois group of $\overline K/K$. Abbes and Saito defined a decreasing filtration $G^r_K$ $(r\in \mathbb Q_{\geq 1})$ of $G_K$, called the ramification filtration \cite{as1,as2}. We fix a prime number $\ell$ different from $p$. Let $M$ be a finitely generated $\mathbb F_{\ell}$-module on which the wild inertia subgroup of $G_K$ acts through a finite quotient, $M=\bigoplus_{r\in \mathbb Q}M^{(r)}$ the slope decomposition of $M$ (cf. \eqref{slopedecom}) and
\begin{equation}\label{introdtmod}
\mathrm{dimtot}_KM=\sum_{r\geq 1}r\cdot\dim_{\mathbb F_{\ell}}M^{(r)}
\end{equation}
the {\it total dimension} of $M$ (cf. \eqref{dtofmod}). If the residue field $F$ is perfect, then the ramification filtration coincides with the canonical upper numbering filtration shifted by one \cite[3.7]{as1} and the total dimension of $M$ is the sum of the classical Swan conductor of $M$ and the dimension of $M$ over $\mathbb F_{\ell}$.

\subsection{}\label{YEW}
In the following, let $k$ be an algebraically closed field of characteristic $p>0$. Let $Y$ be a smooth $k$-scheme, $E$ a reduced effective Cartier divisor on $Y$, $\{E_i\}_{i\in I}$ the set of irreducible components of $E$, $W$ the complement of $E$ in $Y$, $h:W\rightarrow Y$ the canonical injection and $\mathscr G$ a locally constant and constructible sheaf of $\mathbb F_{\ell}$-modules on $W$. Let $\xi_i$ be the generic point of $E_i$, $\bar \xi_i$ a geometric point above $\xi_i$, $\eta_i$ the generic point of the strict localization $Y_{(\bar\xi_i)}$ and $\bar\eta_i$ a geometric point above $\eta_i$. The restriction $\mathscr G|_{\eta_i}$ is associated to a finite $\mathbb F_{\ell}$-module with a continuous $\mathrm{Gal}(\bar\eta_i/\eta_i)$-action.  We denote by $\mathrm{dimtot}_{E_i}(h_!\mathscr G)$ the total dimension of $\mathscr G|_{\eta_i}$ \eqref{introdtmod} and, following \cite[Definition 3.5]{wr}, define the {\it total dimension divisor} of $\mathscr G$ by the formula
\begin{equation}\label{introdefdimtot}
\mathrm{DT}_{Y}(h_!\mathscr G)=\sum_{i\in I}\mathrm{dimtot}_{E_i}(h_!\mathscr G)\cdot E_i.
\end{equation}
It is a divisor with integral coefficients (\cite[4.4.3]{xiao} and \cite[Proposition 3.10]{wr}).

\subsection{}\label{nothigherdim}
In the following, let $S$ be an irreducible $k$-scheme of finite type, $f:X\rightarrow S$ a separated and smooth $k$-morphism of finite type, $D$ an effective Cartier divisor on $X$ relative to $S$, $U$ the complement of $D$ in $X$ and $j:U\rightarrow X$ the canonical injection. For each point $s\in S$, we denote by $\bar s\rightarrow S$ an algebraic geometric point above $s$ and by $X_{\bar s}$ (resp. $D_{\bar s}$) the fiber of $f:X\rightarrow S$ (resp. the fiber of $f|_D:D\rightarrow S$) at $\bar s$. We assume that $f|_D:D\rightarrow S$ is smooth and that, for each geometric point $\bar s\rightarrow S$, the fiber $D_{\bar s}$ is irreducible.  Let  $\mathscr F$ be a locally constant and constructible sheaf of $\mathbb F_{\ell}$-modules on $U$.  For any point $s\in S$, the total dimension number $\mathrm{dimtot}_{D_{\bar s}}(j_!\mathscr F|_{X_{\bar s}})$ is independent of the choice of $\bar s$. It defines a function $\psi:S\rightarrow \mathbb Z$.

\begin{theorem}\label{special statement}
Under the assumptions of \ref{nothigherdim}, the function $\psi:S\rightarrow \mathbb Z$ is constructible and lower semi-continuous on $S$.
\end{theorem}

Confer to theorem \ref{thscdiv} for a more general statement.

\subsection{}
The strategy to prove theorem \ref{special statement} starts from an inequality due to Saito that compares the pull-back of the total dimension divisor of an \'etale sheaf and the total dimension divisor of the pull-back of the sheaf (cf. \ref{saitopbformdt}). We take again the notation and assumptions of \ref{YEW}. Let $C$ be a smooth $k$-curve and $\iota:C\rightarrow Y$ an immersion. We assume that $C$ intersects $E$ properly at a closed point $y\in Y$, that $E$ is smooth at $y$ and that the ramification of $\mathscr G$ is non-degenerate at $y$ \cite[Definition 3.5]{wr} (cf. \ref{recallnd}). Then, we have \cite[Proposition 3.9]{wr}
\begin{equation}\label{intropullback}
(\mathrm{DT}_{Y}(h_!\mathscr G),C)_y\geq\mathrm{dimtot}_y(h_!\mathscr G|_C).
\end{equation}
Moreover, if $\iota:C\rightarrow Y$ satisfies a certain geometric condition (cf. \ref{saitopbformdt}), we have \cite[Proposition 3.9]{wr}
\begin{equation}\label{intropullbackeq}
(\mathrm{DT}_{Y}(h_!\mathscr G),C)_y=\mathrm{dimtot}_y(h_!\mathscr G|_C).
\end{equation}

Firstly, we extend the inequality \eqref{intropullback}. More precisely, we prove that \eqref{intropullback} holds
without assuming that $E$ is smooth at $y$ and that the ramification of $\mathscr G$ is non-degenerate at $y$ (proposition \ref{Results-1}). After that, using the generalized inequality \eqref{intropullback} and the equality \eqref{intropullbackeq}, we reduce theorem \ref{special statement} to a relative curve case. Using theorem \ref{themdelignelaumon}, we deduce  the following two properties of $\psi:S\rightarrow \mathbb Z$:
\begin{itemize}
\item[(1)]
For any closed point $t\in S$, we have $\psi(t)\leq\psi(\eta)$, where $\eta$ denotes the generic point of $S$ (more generally, cf. proposition \ref{geineqprop1}).
\item[(2)]
There exists an open dense subset $V\subset S$ such that, for any closed point $t\in S$, we have $\psi(t)\geq\psi(\eta)$ (more generally, cf. proposition \ref{geineqprop2}).
\end{itemize}
Finally, we prove that theorem \ref{special statement} is a consequence of properties (1) and (2) (cf.  \ref{final proof}).

\subsection{}
This article is organized as follows. We briefly recall Abbes  and Saito's ramification theory and mention some geometric invariants of \'etale sheaves in \S 3. We formulate our main results in \S 4. After presenting some useful geometric preliminaries in \S 5, we prove various generalizations of inequality \eqref{intropullback} (proposition \ref{Results-1} and theorem \ref{mainthmResults}) in \S 6 and \S 7. The last section is devoted to the proof of theorem \ref{special statement} (cf. theorem \ref{thscdiv}).

\subsection*{Acknowledgement}
The authors would like to express their gratitude to professors A. Abbes, H. Esnault, L. Fu and T. Saito for their inspiring discussion and stimulating suggestions. The first author is supported by JSPS postdoctoral fellowship during his stay at the University of Tokyo and the second author is partially supported by Alexander von Humboldt Foundation for his research at Freie Universit\"at Berlin. The authors are grateful to these institutions.

\section{Notation}
\subsection{}
In this article, $k$ denotes an algebraically closed field of characteristic $p>0$. We fix a prime number $\ell$ which is different from $p$ and a finite field $\Lambda$ of characteristic $\ell$. All $k$-schemes are assumed to be separated and of finite type over $\mathrm{Spec}(k)$ and all morphisms between $k$-schemes are assumed to be $k$-morphisms.

\subsection{}\label{Notgeom}
Let $x$ be a point of a scheme $X$. An algebraic geometric point $\bar x\ra X$ above $x$ is a morphism $\mathrm{Spec} (\overline{k(x)})\ra X$ that factors through the canonical map $\mathrm{Spec} (k(x))\ra X$, where $\overline{k(x)}$ is an algebraic closure of $k(x)$.

\subsection{}\label{fiberschdiv}
Let $f:X\rightarrow S$ be a morphism of schemes,  $s\in S$ and $\bar s\rightarrow S$ a geometric point above $s$. We denote by $X_s$ (resp. $X_{\bar s}$) the fiber of $f:X\rightarrow S$ at $s$ (resp. the fiber of $f:X\rightarrow S$ at $\bar s$).

For any morphism $\pi:S'\rightarrow S$, we put $X'=X\times_SS'$ and put $\pi':X'\rightarrow X$ the canonical projection. Assume that $f:X\rightarrow S$ is flat and of finite presentation. Let $D$ be a Cartier divisor on $X$ relative to $S$ \cite[IV, 21.15.2]{EGA4}. We denote by $\pi'^*D$ the pull-back of $D$, which is a Cartier divisor  on $X'$ relative to $S'$  \cite[IV, 21.15.9]{EGA4}. When $S'$ is a point $s'$ of $S$, we simply put $D_{s'}=\pi'^*D$.

\subsection{}
Let $X$ be a smooth scheme over a field. We denote by $\mathbb TX$ the tangent bundle of $X$ and by $\mathbb T^*X$ the cotangent bundle of $X$. For $x\in X$, we put $\mathbb T_xX=\mathbb TX\times_Xx$ and $\mathbb T^*_xX=\mathbb T^*X\times_Xx$.

\subsection{}\label{deftransversal}
Let $\kappa$ be a field, $X$ be a smooth $\kappa$-scheme and $C$ a closed conical subset in $\mathbb T^*X$ \cite[1.0]{beilinson15}.

Let $h:Y\rightarrow X$ be a $\kappa$-morphism of smooth $\kappa$-schemes. We say that $h\colon Y\rightarrow X$  is  \emph{$C$-transversal at a point} $y\in Y$ if $(\ker(dh_y)\cap C)\times_X h(y)=\{0\}\subseteq \mathbb T^*_{h(y)}X$, where $dh_y: \mathbb T^*_{h(y)}X\rightarrow \mathbb T^*_{y}Y$ is the canonical map. We say that  $h\colon Y\rightarrow X$  is  \emph{$C$-transversal} if it is $C$-transversal at every point of $Y$. If $h:Y\rightarrow X$ is $C$-transversal, we define $h^{\circ} C$ to be the image of $Y\times_XC$ in $\mathbb T^*Y$, where $dh: Y\times_X\mathbb T^*X\rightarrow \mathbb T^*Y$ is the canonical map. It is also a closed conical subset of $\mathbb T^*Y$.

Let $f:X\rightarrow W$ be a $\kappa$-morphism of smooth $\kappa$-schemes. We say that $f:X\rightarrow W$ is
\emph{$C$-transversal at a point} $ x\in X$ if $df^{-1}_x(C)=\{0\}\subseteq \mathbb T^*_{f(x)}W$, where $df_x: \mathbb T^*_{f(x)}W\rightarrow \mathbb T^*_xX$ is the canonical map.
We say that $f:X\rightarrow W$ is $C$-{\it transversal} if it is $C$-transversal at every point of $X$.

Let $h:Y\rightarrow X$ and $f:Y\rightarrow W$ be a pair of $\kappa$-morphisms of smooth $\kappa$-schemes. We say that $(h,f)$ is $C$-{\it transversal} if $h:Y\rightarrow X$ is $C$-transversal and $f:Y\rightarrow W$ is $h^{\circ}C$-transversal.

\subsection{}
Let $X$ be a scheme over a field and $D$ and $E$ two Cartier divisors on $X$. We say that $D$ is bigger than $E$ and write $D\geq E$, if $D-E$ is effective.

\subsection{}\label{interproper}
Let $X$ be a scheme over a field of finite type, $f:Y\rightarrow X$ a morphism of finite type and $g:Z\rightarrow X$ a regular immersion of codimension $c$. We say $Z$ {\it intersects}  $Y$ {\it properly} if, for each irreducible component $Y'$ of $Y$,
the fiber product $Z'=Y'\times_XZ$ is equidimensional and the canonical injection $Z'\rightarrow Y'$ is of codimension $c$.

\section{Ramification and geometric invariants of \'etale sheaves}
\subsection{}
Let $K$ be a complete discrete valuation field, $\mathscr O_K$ its integer ring and $F$ the residue field of $\mathscr O_K$. We assume that the characteristic of $F$ is $p>0$. Let $\overline K$ be a separable closure of $K$ and  $G_K$ the Galois group of $\overline K/K$. Abbes and Saito defined a decreasing filtration $G^r_K$ ($ r\in \mathbb Q_{\geq 1}$) of $G_K$ by closed normal subgroups called the ramification filtration \cite[3.1]{as1}. For any $r\geq 1$, we put
\begin{equation*}
G^{r+}_K=\overline{\bigcup_{{s\in \mathbb Q_{>r}}}G^s_K}.
\end{equation*}
The subgroup $G^1_K$ is the inertia subgroup $I_K$ of $G_K$ and $G^{1+}_K$ is the wild inertia subgroup $P_K$ of $G_K$, i.e., the $p$-Sylow subgroup of $I_K$ \cite[3.7]{as1}. If $F$ is perfect, the filtration $\{G^r_K\}_{r\in \mathbb Q_{\geq 1}}$ coincides with the canonical upper numbering filtration shifted by one ({\it loc. cit.}). If $\mathrm{Char}(K)=p$, then, for any $r\in \mathbb Q_{>1}$, the graded piece $G^r_K/G^{r+}_K$ is abelian and killed by $p$ (\cite[2.15]{as2} and \cite[Corollary 2.28]{wr}).

\subsection{}
Let $M$ be a finitely generated $\Lambda$-module with a continuous $P_K$-action. The module $M$ has a unique decomposition \cite[1.1]{katz}
\begin{equation}\label{slopedecom}
M=\bigoplus\limits_{r\geq 1}M^{(r)}
\end{equation}
into $P_K$-stable submodules $M^{(r)}$, such that $M^{(1)}=M^{P_K}$ and for every $r>1$,
\begin{equation*}
(M^{(r)})^{G^r_K}=0\ \ \ {\rm and}\ \ \  (M^{(r)})^{G^{r+}_K}=M^{(r)}.
\end{equation*}
The decomposition \eqref{slopedecom} is called the \emph{slope decomposition} of $M$. The values $r\geq 0$ for which  $M^{(r)}\neq 0$ are called the {\it slopes} of $M$. The \emph{total dimension} of $M$ is defined by
\begin{equation}\label{dtofmod}
\mathrm{dimtot}_K M=\sum_{r\geq 1}r\cdot \mathrm{dim}_{\Lambda}M^{(r)}.
\end{equation}
If the residue field $F$ is perfect, the invariant $\mathrm{dimtot}_K M$ is the classical total dimension of $M$, i.e., $\mathrm{dimtot}_K M=\mathrm{Sw}_KM +\dim_{\Lambda }M$, where $\mathrm{Sw}_KM$ denotes the classical Swan conductor of $M$ \cite[19.3]{linrep}.

\subsection{}\label{gedimtot}
In the following of this section, let $\kappa$ be a field of characteristic $p$.

Let $X$ be a smooth $\kappa$-scheme, $D$ a reduced effective Cartier divisor on $X$, $\{D_i\}_{i\in I}$ the set of irreducible components of $D$, $U$ the complement of $D$ in $X$ and $j:U\rightarrow X$ the canonical injection. We assume that each $D_i$ is generically smooth over $\mathrm{Spec}(\kappa)$. Choose an algebraic closure $\bar\kappa$ of $\kappa$. We denote by $\xi_{i}$ the generic point of an irreducible component of $D_{i,\bar\kappa}=D_i\times_{\kappa}\bar\kappa$, by $\bar\xi_{i}$ a geometric point above $\xi_{i}$, by $\eta_{i}$ the generic point of the strict localization $X_{\bar\kappa,(\bar\xi_{i})}$, by $K_{i}$ the function field of  $X_{\bar\kappa,(\bar\xi_{i})}$ and by $\overline K_i$ a separable closure of $K_i$.
Let $\mathscr F$ be a locally constant and constructible sheaf of $\Lambda$-modules on $U$.
The restriction $\mathscr F|_{\eta_i}$ corresponds to a finitely generated $\Lambda$-module with a continuous $\mathrm{Gal}(\overline K_i/K_i)$-action.
Since the $\mathrm{Gal}(\bar\kappa/\kappa)$-action on the set of irreducible components of $D_{i,\bar\kappa}$ is transitive, the total dimension  $\mathrm{dimtot}_{K_{i}}(\mathscr F|_{\eta_{i}})$ of $\mathscr F|_{\eta_i}$ dose not depend on the choice of $\bar\kappa$ nor on the choice of the irreducible component of $D_{i,\bar\kappa}$. Following \cite[Definition 3.5]{wr}, we define the {\it total dimension divisor} of $j_!\mathscr F$ on $X$ and denote by $\mathrm{DT}_X(j_!\mathscr F)$ the divisor:  \begin{equation}\label{gedtdefinition}
\mathrm{DT}_X(j_!\mathscr F)=\sum_i \mathrm{dimtot}_{K_{i}} (\mathscr F|_{\eta_{i}})\cdot D_i.
\end{equation}
It has integral coefficients (\cite[4.4.3]{xiao} and \cite[Proposition 3.10]{wr}).

Let $\kappa'$ be a finite extension of $\kappa$ and $f:X_{\kappa'}=X\times_{\kappa}\kappa'\rightarrow X$ the canonical projection. By definition, we have
 \begin{equation}\label{basechangetd}
 f^*(\mathrm{DT}_X(j_!\mathscr F))=\mathrm{DT}_{X_{\kappa'}}(f^*j_!\mathscr F).
 \end{equation}

\begin{remark}\label{coefdt}
In the following, under the assumption of \ref{gedimtot}, we denote by $\mathrm{dimtot}_{D_i}(j_!\mathscr F)$ the coefficient of $D_i$ in the total dimension divisor $\mathrm{DT}_X(j_!\mathscr F)$ instead of $\mathrm{dimtot}_{K_i}(\mathscr F|_{\eta_i})$ in \eqref{gedtdefinition}.
\end{remark}

\subsection{}\label{defgdt}
Let $S$ be an irreducible $\kappa$-scheme of finite type, $f:X\rightarrow S$ a smooth $\kappa$-morphism, $\{D_i\}_{i\in I}$ a set of effective Cartier divisors on $X$ relative to $S$, $D$ the sum of all $D_i$ $(i\in I)$, $U$ the complements of $D$ in $X$ and $j:U\rightarrow X$ the canonical injection. For each $i\in I$, we assume that $D_i$ is irreducible as the associated closed subscheme of $X$ and that $f|_{D_i}:D_i\rightarrow S$ is smooth. We denote by $\eta$ the generic point of $S$.
Notice that $D_{i,\eta}$ is an integral effective Cartier divisor on $X_{\eta}$ for each $i\in I$.
Let $\mathscr F$ be a locally constant and constructible sheaf of $\Lambda$-modules on $U$. We define the {\it generic total dimension divisor of} $\mathscr F$ on $X$ and denote by $\mathrm {GDT}_f(j_!\mathscr F)$ the divisor (remark \ref{coefdt})
\begin{equation}\label{gdtdefform}
\mathrm{GDT}_f(j_!\mathscr F)=\sum_{i\in I}\mathrm{dimtot}_{D_{i,\eta}}(j_!\mathscr F|_{X_{\eta}})\cdot D_i.
\end{equation}

Let $S'$ be an irreducible smooth $\kappa$-scheme of finite type and $\pi:S'\rightarrow S$ a dominant and generically finite $\kappa$-morphism. We denote by
$f':X'\rightarrow S'$ (resp. by $\pi':X'\rightarrow X$) the base change of $f:X\rightarrow S$ (resp. $\pi:S'\rightarrow S$). By \eqref{basechangetd}, we have
\begin{equation}\label{finitepbgdt}
\pi'^*(\mathrm{GDT}_f(j_!\mathscr F))=\mathrm{GDT}_{f'}(\pi'^*j_!\mathscr F).
\end{equation}

\subsection{}\label{recallnd}
Let $X$ be a connected smooth $k$-scheme of dimension $n$, $D$ a reduced effective Cartier divisor on $X$, $U$ the complement of $D$ in $X$ and $j:U\rightarrow X$ the canonical injection. Let $\mathscr F$ be a locally constant and constructible sheaf on $U$. After removing a closed subset $Z\subseteq D$ of codimension $2$ in $X$, the complement $D_0=D-Z$ is smooth and the ramification of $\mathscr F$ along $D_0$ is {\it non-degenerate} \cite[Definition 3.1]{wr}. Roughly speaking, the ramification of an \'etale sheaf along a divisor is non-degenerate if its ramification is controlled by the ramification at generic points of the divisor. Saito defined the {\it characteristic cycle} of $j_!\mathscr F$ on the complement $X_0=X-Z$ using Abbes and Saito's ramification theory and denoted by $CC(j_!\mathscr F)$ \cite[Definition 3.5]{wr}. It is an $n$-cycle on the cotangent bundle $\mathbb T^*X_0$. The support of $CC(j_!\mathscr F)$ is called the {\it singular support} of $j_!\mathscr F$ and denoted by $SS(j_!\mathscr F)$. It is a closed conical subset of $\mathbb T^*X_0$. In this case, the singular support $SS(j_!\mathscr F)$ is a union of the zero-section $\mathbb T^*_{X_0}X_0$ and a closed conical subset of $D_0\times_{X_0}\mathbb T^*X_0$ and, for every point $x\in D_0$, the fiber $SS_x(j_!\mathscr F)=SS(j_!\mathscr F)\times_{X}x$ is of dimension $1$ \cite[Proposition 3.12]{wr}.

\subsection{}\label{saitopbformdt}
We take the notation and assumptions of \ref{recallnd}. Let $C$ be a smooth $k$-curve and $g:C\rightarrow X$ an immersion. We assume that $C$ intersects $D$ properly at a closed point $x\in X$, that $D$ is smooth at $x$ and that the ramification of $\mathscr F$ along the divisor $D$ is non-degenerate at $x$. Hence, the singular support of $j_!\mathscr F$ can be defined using ramification theory in a neighborhood of  $x$ (cf. \ref{recallnd}). Then we have \cite[Proposition 3.9]{wr}
\begin{equation}\label{saitopbineq}
(\mathrm{DT}_X(j_!\mathscr F), C)_x\geq\mathrm{dimtot}_x (j_!\mathscr F|_C).
\end{equation}
Moreover, if $g:C\rightarrow X$ is $SS(j_!\mathscr F)$-transversal at $x$ (\ref{deftransversal}), then \cite[Proposition 3.9]{wr}
\begin{equation}\label{pullbacktocurveundernonchar}
(\mathrm{DT}_X(j_!\mathscr F), C)_x=\mathrm{dimtot}_x (j_!\mathscr F|_C).
\end{equation}

\subsection{}\label{transveralandss}
We should mention that the singular support and characteristic cycle of \'etale sheaves can be constructed in a more general setting beyond \cite{wr} (cf. \ref{recallnd}).

Let $X$ be a smooth $\kappa$-scheme and $\mathscr K$ be a complex of \'etale sheaves of $\Lambda$-modules on $X$ with bounded constructible cohomologies. We say that $\mathscr K$ is {\it micro-supported} on a closed conical subset $C\subseteq \mathbb T^*X$ if for every $C$-transveral pair of $\kappa$-morphisms $h:Y\rightarrow X$ and $f:Y\rightarrow W$ (\ref{deftransversal}), the morphism $f:Y\rightarrow W$  is locally acyclic with respect to $h^*\mathscr K$. Beilinson recently proved that there exists a smallest closed conical subset $C\subseteq \mathbb T^*X$ on which $\mathscr K$ is micro-supported \cite[Theorem 1.2]{beilinson15}. It is called the {\it singular support} of $\mathscr K$ and denoted by $SS(\mathscr K)$. If $X$ is connected, each irreducible component of $SS(\mathscr K)$ has dimension $\dim_{\kappa}X$ ({\it loc. cit.}).

Assume that $\kappa$ is perfect and that $X$ is connected of dimension $n$. Saito recently defined an $n$-cycle in $\mathbb T^*X$ supported on the singular support of $\mathscr K$, which is characterized by a Milnor formula \cite[Theorem 3.6]{2015}. It is called the {\it characteristic cycle} of $\mathscr K$ and denoted by $CC(\mathscr K)$. The intersection of $CC(\mathscr K)$ and the zero-section of $\mathbb T^*X$ computes the Euler-Poincar\'e characteristic of $\mathscr K$ \cite[Theorem 4.21]{2015}.

\section{Main results}

\subsection{}\label{xdf}
Let $X$ be a smooth $k$-scheme, $D$ a reduced effective Cartier divisor on $X$, $U$ the complement of $D$ in $X$, $j:U\rightarrow X$ the canonical injection and $\mathscr F$ a locally constant and constructible sheaf of $\Lambda$-module.

\begin{proposition}\label{Results-1}
We take the notation and assumptions of \ref{xdf}. Let $C$ be a smooth $k$-curve and $i:C\rightarrow X$ an immersion. We assume that $C$ intersects $D$ properly at a closed point $x\in X$.  Then, we have
\begin{equation}\label{mainineq}
  (\mathrm{DT} (j_!\sF),C)_x\geq \mathrm{dimtot}_x(j_!\sF|_C).
 \end{equation}
\end{proposition}
This proposition is a generalization of  \eqref{saitopbineq}. More generally, we have the following theorem~:

\begin{theorem}\label{mainthmResults}
We take the notation and assumptions of \ref{xdf}. Let $Y$ be a smooth $k$-scheme and $f:Y\ra X$ a $k$-morphism such that $Y\times_XD$ is an effective Cartier divisor on $Y$. Then, we have
 \begin{equation}\label{high}
    f^*(\mathrm{DT}_X (j_!\mathscr F))\geq \mathrm{DT}_Y (f^*j_!\sF).
  \end{equation}
\end{theorem}

\subsection{}\label{notsctd}
Let $S$ be an irreducible $k$-scheme, $f:X\rightarrow S$ a smooth $k$-morphism, $\{D_i\}_{i\in I}$ a set of effective Cartier divisors on $X$ relative to $S$, $D$ the sum of all $D_i$ $(i\in I)$, $U$ the complements of $D$ in $X$ and $j:U\rightarrow X$ the canonical injection. For every $i\in I$, we assume that $D_i$ is irreducible as the associated closed subscheme of $X$ and that $f|_{D_i}:D_i\rightarrow S$ is smooth. For any point $t\in S$, we denote  by $\rho_t:X_t\rightarrow X$ the canonical injection. Let $\mathscr F$ be a locally constant and constructible sheaf of $\Lambda$-modules on $U$.

\begin{theorem}\label{thscdiv}
Under the notation and assumptions of \ref{notsctd}, there exists an open dense subset $V\subseteq S$ such that:
\begin{itemize}
\item[(1)]
for any point $s\in V$, we have $\rho_s^*(\mathrm{GDT}_f(j_!\mathscr F))=\mathrm{DT}_{X_{s}}(j_!\mathscr F|_{X_{s}})$;
\item[(2)]
for any point $t\in S-V$, we have $\rho_t^*(\mathrm{GDT}_f(j_!\mathscr F))\geq\mathrm{DT}_{X_{t}}(j_!\mathscr F|_{X_{t}})$.
\end{itemize}
\end{theorem}
This theorem generalizes Deligne and Laumon's result \cite[2.1.1]{lau} (cf. theorem \ref{themdelignelaumon}) in the geometric case.

\section{Geometric preliminaries}

\begin{proposition}[cf. {\cite[I, Chapitre 0, 15.1.16]{EGA4}}]\label{flatlemma}
Let $A\rightarrow B$ be a local homomorphism of noetherian local rings, $\kappa$ the residue field of $A$, $b$ an element of the maximal ideal of $B$ and $\bar b\in B\times_A\kappa$ the residue class of $b$. Then, the following conditions are equivalent:
\begin{itemize}
\item[(1)]
The quotient $B/bB$ is flat over $A$ and $b$ is a non-zero divisor of $B$;
\item[(2)]
$B$ is flat over $A$ and $\bar b$ is a non-zero divisor of $B\otimes_A\kappa$.
\end{itemize}
\end{proposition}
It is deduced by the equivalence of (b) and (c) of \cite[I, Chaptire 0, 15.1.16]{EGA4}.

\begin{proposition}[{\cite[IV, 18.12.1]{EGA4}}]\label{zarmain}
Let $S$ and $D$ be schemes, $f:D\rightarrow S$ a separated morphism locally of finite type, $x$ a point of $D$ and $s=f(x)$. We assume that $x$ is an isolated point of $f^{-1}(s)$. Then, there exists an \'etale morphism $S'\rightarrow S$, a point $x'$ of $D'=D\times_SS'$ above $x$ and a Zariski open and closed neighborhood $V'$ of $x'\in D'$ such that  $V'$ is finite over $S'$ and $f'^{-1}(f'(x'))\cap V'=\{x'\}$, where $f':D'\rightarrow S'$ denotes the base change of $f:D\rightarrow S$.
\end{proposition}

\begin{proposition}\label{transversalcurve}
Let $\kappa$ be a field with infinitely many elements, $X$ a connected smooth $\kappa$-scheme of dimension $n\geq 2$, $D$ an effective Cartier divisor on $X$ which is smooth at a $\kappa$-rational point $x\in D$ and $S\subseteq \mathbb T^*_xX$ a closed conical subset of dimension $1$. Then, we can find a smooth $\kappa$-curve $C$ and an immersion $h:C\rightarrow X$ such that $C$ intersects $D$ transversally at $x$ and that $\ker(dh)\cap S=\{0\}$, where $dh:\mathbb T^*_xX\rightarrow\mathbb T^*_xC$ is the canonical map.
\end{proposition}
\begin{proof}
Since $S$ is a closed conical subset of dimension $1$ of $\mathbb T^*_xX$, it is a union of finite many $1$-dimensional vector spaces $\bigcup_{i\in I}L_i\subseteq \mathbb T^*_x X$. We denote by $L^{\vee}_i\subseteq\mathbb T_xX$ the dual hyperplane of $L_i$. The divisor $D$ is smooth at $x\in X$, then the stalk of the tangent bundle $\mathbb T_xD$ is a hyperplane of $\mathbb T_xX$. Since $\kappa$ has infinitely many elements, we can find a non-zero vector $\lambda\in\mathbb T_xX-((\bigcup_{i\in I}L^{\vee}_i)\bigcup\mathbb T_xD)$.

We denote by $H_{\lambda}\subseteq \mathbb T^*_x X$ the dual hyperplane corresponding to $\lambda \in\mathbb T_xX$ and by $\bar t_1,\cdots,\bar t_{n-1}$ generators of $H_{\lambda}$. Let $t_1,\cdots,t_{n-1}$ be liftings of $\bar t_1,\cdots,\bar t_{n-1}$ in the maximal idea $\mathfrak m_{x}$ of $\mathscr O_{X,x}$. Zariski locally on $x\in X$, we define a curve $C$ by the ideal $(t_1,\cdots, t_{n-1})$. It is smooth at $z$ and $\mathbb T_xC$ is spanned by $\lambda\in \mathbb T_xX$. We have $(C,D)_x=1$ since $\mathbb T_xC\nsubseteq\mathbb T_xD$. The relation $\mathbb T_xC\nsubseteq \bigcup_{i\in I}L^{\vee}_i$ implies that $\ker(dh)\cap S=\{0\}$.
\end{proof}

\begin{proposition}[cf. {\cite[III, 9.5.3]{EGA4}}]\label{genetosp2}
Let $S$ be an irreducible noetherian scheme, $g:D\rightarrow S$ a morphism of finite type, $\{D_i\}_{i\in I}$  the set of irreducible components of $D$. We assume that, for each $i\in I$, the restriction $f|_{D_i}:D_i\rightarrow S$ is surjective. Then there exists an open dense subset $V\subseteq S$ such that, for every point $v\in V$ and for any indices $i, j\in I$ $(i\neq j)$, the fibers $D_{i,v}$ and $D_{j,v}$ do not have common irreducible components.
\end{proposition}
\begin{proof}
Without loss of generality, we may assume that $D$ has only two irreducible components $D_1$ and $D_2$. We denote by $D_i^{\circ}$ $(i=1,2)$ the complement of $D_1\cap D_2$ in $D_i$. Let $\eta$ be the generic point of $S$. For each $i=1,2$, the canonical inclusion $D_{i,\eta}^{\circ}\subseteq D_{i,\eta}$ is dense since $D_{i,\eta}$ is irreducible. By \cite[III, 9.5.3]{EGA4}, there exists an open dense subset $V\subseteq S$, such that, for any point $v\in V$ and each $i=1,2$, the canonical inclusion $D^{\circ}_{i,v}\subseteq D_{i,v}$ is dense. It is equivalent to say that, for any $v\in V$, the fibers $D_{1,v}$ and $D_{2,v}$ do not have common irreducible components.
\end{proof}

\begin{proposition}[cf. {\cite[III, 9.7.7]{EGA4}}]\label{genetosp1}
Let $S$ and $D$ be integral $k$-schemes and $f:D\rightarrow S$ a smooth $k$-morphism. Then there exists an irreducible $k$-scheme $W$ and an \'etale map $h:W\rightarrow S$ such that,
\begin{itemize}
\item[(1)]
Each connected component of $D\times_SW$ is irreducible;
\item[(2)]
Every connected component of  $D\times_SW$ has geometrically irreducible fibers over $W$.
\end{itemize}
\end{proposition}
\begin{proof}
We denote by $\eta$ the generic point of $S$. There exists a point $\eta'$ and a finite \'etale morphism $h_0:\eta'\rightarrow \eta$ such that every irreducible component of $D_{\eta'}=D\times_S\eta'$ is geometrically irreducible. We denote by $S'$ the normalization of $S$ in $\eta'$. Since the canonical map $D'=D\times_SS'\rightarrow S'$ is flat,  the image of the generic point of each irreducible component of $D'$ in $S'$ is $\eta'$, i.e., there is a one to one correspondence between the irreducible components of $D'$ and those of $D_{\eta'}$. After replacing $S'$ by a Zariski neighborhood $W'$ of $\eta'$, we may assume that $W'$ is a smooth $k$-scheme and \'etale over $S$. Then, each connected component of $D\times_SW'$ is a smooth $k$-scheme, particularly, is irreducible.
We denote by $\{E_r\}_{r\in J}$ the set of connected components of $D\times_SW'$.
 By \cite[III, 9.7.7]{EGA4}, for each $r\in J$, there exists an Zariski open dense subset $W_r\subset W'$ such that,  for every $w\in W_r$, the fiber $E_{r,w}$ is geometrically irreducible. We let $W$ be the intersection of all $W_r$ $(r\in J)$.
\end{proof}

\begin{corollary}\label{compcycle}
Let $S$ be an irreducible $k$-scheme, $f:X\rightarrow S$ a smooth $k$-morphism, $\{D_i\}_{i\in I}$ a set of effective Cartier divisors on $X$ relative to $S$, $D$ the sum of all $D_i$ $(i\in I)$. For every $i\in I$, we assume that $D_i$ is irreducible as the associated closed subscheme of $X$ and that $f|_{D_i}:D_i\rightarrow S$ is smooth.
 Let $R=\sum_{i\in I} r_i\cdot D_i$ be a Cartier divisor on $X$. If there exists an open dense subset $V\subseteq S$ such that, for each closed point $t\in V$, the fiber $R_t$ is an effective Cartier divisor on $X_t$ (\ref{fiberschdiv}), then $R$ is also effective.
\end{corollary}
\begin{proof}
We may assume that $S$ is integral. By proposition \ref{genetosp2}, this problem can be  reduced to the case where $D$ is irreducible. Let $h:W\rightarrow S$ be an \'etale morphism and we denote by $h_W:X_W=X\times_SW\rightarrow X$ its base change. Then $R$ is effective if and only if
$h^*_WR$ is effective. Hence, replacing $S$ by an \'etale neighborhood, we may assume that, for any point $s\in S$, the fiber $D_s$ is geometrically irreducible (proposition \ref{genetosp1}). In this case, for any closed point $s\in S$ such that $D_s\neq \emptyset$, the coefficient of $D_s$ in $R_s$ equals the coefficient of $D$ in $R$.
\end{proof}

\begin{proposition}[{\cite[Proposition 3.8]{wr}}]\label{projection}
Let $X$ be a smooth $k$-scheme, $D$ a reduced Cartier divisor on $X$, $U$ the complement of $D$ in $X$, $j:U\rightarrow X$ the canonical injection and $\mathscr F$ a locally constant and constructible sheaf of $\Lambda$-modules on $U$. Then for any smooth morphism $\pi: Y\ra X$, we have
 \begin{equation}
    \pi^*(\mathrm{DT}_X(j_!\sF))=\mathrm{DT}_Y (\pi^*j_!\sF).
  \end{equation}
\end{proposition}

\section{Bounding the pull-back of the total dimension divisor over a curve}\label{smoothcurvecase}
\subsection{}\label{genot}
In this section,  we prove proposition \ref{Results-1}.  We take the notation and assumptions of \ref{xdf} and proposition \ref{Results-1}.

\subsection{}\label{smf}
Proposition \ref{Results-1} is a local statement for the Zariski topology on $X$. We assume that $X$ is connected of dimension $n\geq 2$. Choose a regular system of parameters $t_1,\cdots, t_n$ of $\mathscr O_{X,x}$ such that the ideal $(t_1,\cdots,t_{n-1})$ is the kernel of canonical map $i^*:\mathscr O_{X,x}\rightarrow\mathscr O_{C,x}$. We defined a map $f_x:k[T_1,\dots, T_{n-1}]\rightarrow \mathscr O_{X,x}$ by
\begin{equation*}
f_x\colon k[T_1,\dots, T_{n-1}]\rightarrow \mathscr O_{X,x},\ \ \ T_r\mapsto t_r.
\end{equation*}
In a Zariski neighborhood of $x\in X$, it induces a $k$-morphism  $f:X\ra \bA^{n-1}_k$ which is smooth at $x$. Since $C$ intersects $D$ properly at $x$, the restriction $f|_D:D\rightarrow \mathbb A^{n-1}_k$ is flat and quasi-finite at $x\in X$ (proposition \ref{flatlemma}). After shrinking $X$, we may assume that the map $f$ is smooth and the restriction $f|_D:D\ra\bA^{n-1}_k$ is flat and quasi-finite. Notice that $C=f^{-1}(O)$ where $O\in\mathbb A^{n-1}_k$ denotes the origin.

\subsection{}
By \cite[IV, 18.12.1]{EGA4} (cf. proposition \ref{zarmain}), there exists a connected \'etale neighborhood $\rho:V\rightarrow \mathbb A^{n-1}_k$ of $O$ and a point $x'\in D_V=D\times_{\mathbb A^{n-1}_k}V$ above $x$ such that $D_V$ is a disjoint union of two schemes such that one of them, denoted by $R$, is finite and flat over $V$ and $x'$ is the unique point in the intersection $f_V^{-1}(O')\cap R$, where $f_V:X_V\rightarrow V$ denotes the base change of $f:X\rightarrow \mathbb A^{n-1}_k$ and $O'=f_V(x')$. After replacing $X_V$ by a Zariski neighborhood $W$ of $x'$, we may assume that $R=D_V\cap W$. We denote by $g:W\rightarrow V$ the composition of the canonical injection $W\rightarrow X_V$ and $f_V:X_V\rightarrow V$. We have the following commutative diagram
\begin{equation}
\xymatrix{\relax
R\ar[d]\ar[r]\ar@{}|-(0.5){\Box}[rd]&W\ar[r]^g\ar[d]^{\pi}&V\ar[d]^{\rho}\\
D\ar[r]&X\ar[r]_-(0.5)f&\mathbb A^{n-1}_k}
\end{equation}
where the left square is Cartesian.

Let $Z\subseteq R$ be a closed subset of codimension $1$ in $R$ such that $R-Z$ is smooth over $\mathrm{Spec}(k)$ and the ramification of $\mathscr F|_W$ along $R-Z$ is non-degenerate (\ref{recallnd}). Since $g|_R:R\rightarrow V$ is finite and flat, the image $g(Z)$ is a closed subset of $V$, which is not dense in $V$. Then there exists a smooth $k$-curve $B$ and an immersion $\iota:B\rightarrow V$ such that $O'\in \iota(B)$ and that $\iota(B)\cap g(Z)\subseteq\{O'\}$.

\subsection{}
For any closed point $b\in B-\{O'\}$, we have \cite[Proposition 3.9]{wr} (cf. \eqref{saitopbineq})
\begin{equation}\label{ineq-main}
\deg(\mathrm{DT}_{W}( j_!\sF|_{W}), g^{-1}(b))\geq\sum_{y\in(R\cap g^{-1}(b))_{\mathrm{red}}}\mathrm{dimtot}_y(j_!\mathscr F|_{g^{-1}(b)}).
\end{equation}
By Deligne and Laumon's lower semi-continuity of Swan conductors \cite[2.1.1]{lau}, the function
\begin{equation*}
F\colon B \rightarrow \mathbb Z,\ \ \ b\mapsto \sum_{y\in(R\cap g^{-1}(b))_{\mathrm{red}}}\mathrm{dimtot}_y(j_!\mathscr F|_{g^{-1}(b)}).
\end{equation*}
is lower semi-continuous. Meanwhile, the function
\begin{equation*}
G:B\rightarrow \mathbb Z,\ \ \ b\mapsto \deg(\mathrm{DT}_{W}( j_!\sF|_{W}), g^{-1}(b))
\end{equation*}
is constant, since $\mathrm{DT}_{W}( j_!\sF|_{W})$ is a divisor on $W$ supported on $R$ and $g|_R:R\rightarrow V$ is finite and flat. We deduce form \eqref{ineq-main} that $G(O')\geq F(O')$, i.e.,
\begin{equation}\label{smooth-eq-1}
(\mathrm{DT}_{W} (j_!\sF|_{W}), g^{ -1}(O'))_{x'}\geq \mathrm{dimtot}_{x'}(j_!\sF|_{g^{ -1}(O')})=\mathrm{dimtot}_x(j_!\sF|_{C}).
\end{equation}
By proposition \ref{projection} and the flat pull-back formula, we have
\begin{eqnarray}
(\mathrm{DT}_{W} (j_!\sF|_{W}), g^{-1}(O'))_{x'} &=&(\pi^* (\mathrm{DT}_X (j_!\mathscr F)), \pi^*C)_{x'} \label{smooth-eq-3}\\
\nonumber &=&(\mathrm{DT}_X (j_! \mathscr F), C)_x.
\end{eqnarray}
By \eqref{smooth-eq-1} and \eqref{smooth-eq-3}, we obtain \eqref{mainineq}
\begin{equation*}
 (\mathrm{DT}_X (j_!\mathscr F), C)_x\geq \mathrm{dimtot}_x(j_!\mathscr F|_C).
\end{equation*}
\hfill$\Box$

\section{Bounding the pull-back of the total dimension divisor over an arbitrary variety}\label{generalcase}
\subsection{}
In this section, we prove theorem \ref{mainthmResults} and give a corollary. We take the notation and assumptions of \ref{xdf} and theorem \ref{mainthmResults}.

\subsection{}{\it proof of theorem \ref{mainthmResults}}.

This is a local problem for the Zariski topology of $Y$. We may assume that $f:Y\rightarrow X$ is the composition of a closed immersion $i:Y\rightarrow \mathbb A^n_X$ and the canonical projection $\pi:\mathbb A^n_X\rightarrow X$. By proposition \ref{projection}, we have
\begin{equation*}
\pi^*(\mathrm{DT}_X(j_!\sF))=\mathrm{DT}_{\mathbb A^n_X} (\pi^*j_!\sF).
\end{equation*}
Hence, we are reduced to the case where $f:Y\rightarrow X$ is a closed immersion.  The case $\dim_k Y=1$ is done in section \ref{smoothcurvecase}. We focus on the case where $\dim_kY\geq 2$.

After replacing $Y$ by a neighborhood of a generic point of $S=(Y\times_XD)_{\mathrm{red}}$, we may assume that $S$ is  irreducible and smooth over $\mathrm{Spec}(k)$ and that the ramification of $f^*j_!\sF$ along $S$ is non-degenerate (\ref{recallnd}).  Let $y$ be a closed point of $S$. The fiber $SS_y(f^\ast j_!\mathscr F)\in \mathbb T^*_y Y$ of the singular support of $f^*j_!\mathscr F$ is of dimension $1$. By proposition \ref{transversalcurve}, there exists a smooth $k$-curve $C$ and a closed immersion $i:C\rightarrow Y$ such that $C$ intersects $S$ transversally  at $y$ and that $i:C\rightarrow Y$ is $SS(\mathscr f^*j_!\mathscr F)$-transversal at $y$.

By \cite[Proposition 3.9]{wr} (cf. \eqref{pullbacktocurveundernonchar}), we have
\begin{equation}\label{cutbycurve0}
(\mathrm{DT}_Y (f^*j_!\sF),C)_y=\mathrm{dimtot}_y(j_!\sF|_C).
\end{equation}
We put $x=f(y)$. Consider $C$ as a curve in $X$ by the closed immersion $f\circ i:C\rightarrow X$. By proposition \ref{Results-1}, we have
\begin{equation}\label{cutbycurve1}
(\mathrm{DT}_X (j_!\sF),C)_x\geq\mathrm{dimtot}_x(j_!\sF|_C).
\end{equation}
Since $f$ is a closed immersion, we have
\begin{equation}\label{cutbycurve2}
(\mathrm{DT}_X (j_!\sF),C)_x=(f^*(\mathrm{DT}_X (j_!\sF)),C)_y.
\end{equation}
By \eqref{cutbycurve0}, \eqref{cutbycurve1} and \eqref{cutbycurve2}, we have
\begin{equation}\label{pbintersection}
(f^*(\mathrm{DT}_X (j_!\sF)),C)_y\geq (\mathrm{DT}_Y (f^*j_!\sF),C)_y.
\end{equation}
Since $C$ intersects $S$ transversally at $y$, the inequality \eqref{pbintersection} implies that
\begin{equation*}
f^*(\mathrm{DT}_X (j_!\sF))\geq \mathrm{DT}_Y (f^*j_!\sF).
\end{equation*}
\hfill$\Box$

\subsection{}\label{blowupnot}
Let $z$ be a closed point of $D$, $\widetilde X$ the blow-up of $X$ at $z$, $g:\widetilde X\rightarrow X$ the canonical projection, $\widetilde j:U\rightarrow \widetilde X$ the unique lifting of $j:U\rightarrow X$, $\{D_i\}_{i\in I}$ the set of irreducible components of $D$ that contains $z$, $\widetilde D_i$ $(i\in I)$ the strict transform of $D_i$ in $\widetilde X$ and $E=g^{-1}(z)$ the exceptional divisor on $\widetilde X$. We denote by $e_z (D_i)$ $(i\in I)$ the multiplicity of $D_i$ at $z$ \cite[4.3]{ful}.

\begin{corollary}\label{dtbth}
We have the following inequality (remark \ref{coefdt})
\begin{equation}\label{dimtotblow}
  \sum_{i\in I}e_z(D_i)\cdot\mathrm{dimtot}_{D_i}( j_!\mathscr F)\geq \mathrm{dimtot}_{E}(\widetilde j_!\mathscr F).
\end{equation}
\end{corollary}

\begin{proof}
Applying theorem \ref{mainthmResults} to the morphism $g:\widetilde X\rightarrow X$, we have
\begin{equation}\label{inbu}
g^*(\mathrm{DT}_X(j_!\mathscr F))\geq\mathrm{DT}_{\widetilde X}(\widetilde j_!\mathscr F).
\end{equation}
Since $g^*D_i=\widetilde D_i+e_z (D_i)\cdot E$ \cite[Example 4.3.9]{ful}, the coefficient of $E$ in the left hand side of \eqref{inbu} is $\sum_{i\in I}e_z(D_i)\cdot\mathrm{dimtot}_{D_i}(j_!\mathscr F)$. The coefficient of $E$ in the right hand side of \eqref{inbu} is $ \mathrm{dimtot}_{E}(j_!\mathscr F)$. Hence, we get \eqref{dimtotblow}.
\end{proof}

\section{Semi-continuity of total dimension divisors}\label{sctd}
\subsection{}
In this section, we prove theorem \ref{thscdiv}. We take the notation and assumptions of \ref{notsctd}. If $f:X\rightarrow S$ is a relative curve, then theorem \ref{thscdiv} is just a consequence of  \cite[2.1.1]{lau} (cf. theorem \ref{themdelignelaumon}). We focus on the case where $\dim_k X-\dim_kS\geq 2$.

\begin{proposition}\label{geineqprop1}
For each closed point $s\in S$, we have
\begin{equation}\label{geineqform1}
\rho_s^*(\mathrm{GDT}_f(j_!\mathscr F))\geq\mathrm{DT}_{X_{s}}(j_!\mathscr F|_{X_{s}}).
\end{equation}
\end{proposition}

\begin{proof}
We may assume that $S$ is integral. By de Jong's alteration \cite[4.1]{dej}, there exists a smooth $k$-scheme $S'$ and a proper, surjective and generically finite morphism $\alpha:S'\rightarrow S$. Let $f':X'\rightarrow S'$ be the base change of $f:X\rightarrow S$, $s\in S$ a closed point, $s'\in S'$ a closed point above $s$ and $\rho'_{s'}:X'_{s'}\rightarrow X'$ the canonical injection. Notice that $X'_{s'}\xrightarrow{\sim} X_s$.
By \eqref{finitepbgdt}, to show \eqref{geineqform1}, it is sufficient to show that
\begin{equation*}
\rho'^*_{s'}(\mathrm{GDT}_{f'}(j_!\mathscr F|_{X'}))\geq \mathrm{DT}_{X_{s'}}(j_!\mathscr F|_{X'_{s'}}).
\end{equation*}
Hence, we are reduced to the case where $S$ is smooth over $\mathrm{Spec}(k)$.

In the following, we fix a closed point $s\in S$. This is a local statement for the Zariski topology of $X$. After shrinking $X$, we may assume that $X$ is irreducible, that $(D_{s})_{\mathrm{red}}$ has a unique irreducible component and that, for each $i\in I$, $(D_s)_{\mathrm{red}}\subseteq D_i$. Notice that, for each $i\in I$, we have  $D_{i,s}\xrightarrow{\sim} (D_s)_{\mathrm{red}}$. We put  $n=\dim_k X-\dim_k S$.

Let $z\in D_s$ be a closed point such that the ramification of $j_!\mathscr F|_{X_s}$ along $(D_{s})_{\mathrm{red}}$ is non-degenerate at $z$ (\ref{recallnd}). We can find a smooth $k$-curve $C$ and a closed immersion $\iota:C\rightarrow X_s$ such that
the curve $C$ intersects $(D_{s})_{\mathrm{red}}$ transversally at $z$ and that the immersion $\iota:C\rightarrow X_s$ is $SS(j_!\mathscr F|_{X_s})$-transversal at $z$ (proposition \ref{transversalcurve}) .
Choose a regular system of parameters $\bar t_1,\bar t_2,\cdots, \bar t_n$ of $\mathscr O_{X_s,z}$, such that $(\bar t_1,\cdots,\bar t_{n-1})$ is the kernel of $\mathscr O_{X_s,z}\rightarrow \mathscr O_{C,z}$  and that $(\bar t_n)$ is the kernel of $\mathscr O_{X_s,z}\rightarrow \mathscr O_{(D_{s})_{\mathrm{red}},z}$. For each $1\leq r\leq n-1$, choose a lifting $t_r\in \mathscr O_{X,z}$ of $\bar t_r\in \mathscr O_{X_s,z}$. We define an
$\mathscr O_{S,s}$-homomorphism $ g_z:\mathscr O_{S,s}[T_1,\cdots,T_{n-1}]\rightarrow \mathscr O_{X, z}$ by

\begin{equation*}
g_z:\mathscr O_{S,s}[T_1,\cdots T_{n-1}]\rightarrow \mathscr O_{X, z},\ \ \ T_r\mapsto t_r.
\end{equation*}
After replacing $X$ by a Zariski neighborhood of $z$, the map $g_z$ induces an $S$-morphism $g:X\rightarrow  \mathbb A^{n-1}_S$. It satisfies the following conditions after shrinking further $X$
\begin{itemize}
\item[(i)]
it is smooth and of relative dimension $1$;
\item[(ii)]
for each $i\in I$, the restriction $g|_{D_i}:D_i\rightarrow \mathbb A^{n-1}_S$ is \'etale;
\item[(iii)]
the restriction $g|_{D}:D\rightarrow A^{n-1}_S$ is quasi-finite and flat (proposition \ref{flatlemma});
\item[(iv)]
the curve $C$ is the pre-image $g^{-1}(s\times O)$, where $s\times O$ denotes product of $s\in S$ and the origin $O\in\mathbb A^{n-1}_k$.
\end{itemize}

Let $\sigma:S\rightarrow \mathbb A^{n-1}_S$ be the zero-section of the canonical projection $\pi:\mathbb A^{n-1}_S\rightarrow S$. We denote by $h:Y\ra S$ the base change of $g:X\rightarrow  \mathbb A^{n-1}_S$ by $\sigma:S\rightarrow \mathbb A^{n-1}_S$ and we put $E=Y\times_X D$ and, for $i\in I$, put $E_i=Y\times_XD_i$. By proposition \ref{zarmain}, there exists a connected  \'etale neighborhood $\gamma:S'\ra S$ of $s\in S$ such that
\begin{itemize}
\item[(1)]
the pre-image $s'=\gamma^{-1}(s)$ is a point;
\item[(2)]
the fiber product $E'=E\times_SS'$ is a disjoint union of two schemes $Q'$ and $T'$,
\item[(3)]
the canonical map $Q'\rightarrow S'$ is finite and flat and $Q'_{s'}=\{z'\}$, where $z'$ denotes the unique pre-image of $z\in E_s$.
\end{itemize}
Since the restriction $h|_{E_i}:E_i\ra S$ is \'etale for each $i\in I$, we may also assmue that
\begin{itemize}
\item[(4)]
for each $i\in I$, the canonical map $Q'_i=E_i\times_EQ'\rightarrow S'$ is an isomorphism.
\end{itemize}
We have the following commutative diagram with Cartesian squares
\begin{equation*}
\xymatrix{\relax
Q'\ar[r]&E'\ar[r]\ar@{}|-(0.5){\Box}[rd]\ar[d]&E\ar@{}|-(0.5){\Box}[rd]\ar[d]\ar[r]&D\ar[d]&\\
&Y'\ar[d]_{h'}\ar[r]\ar@{}|-(0.5){\Box}[rd]&Y\ar[d]^{h}\ar[r]\ar@{}|-(0.5){\Box}[rd]&X\ar[d]^{g}&\\
&S'\ar[r]^{\gamma}&S\ar[r]^-(0.5){\sigma}&\mathbb A^{n-1}_S\ar[r]^-(0.5){\pi}&S}
\end{equation*}

Let $\eta$ be the generic point of $S$ and $\bar\eta$ an algebraic geometric point above $\eta$ that factors through $S'$. Recall that $X_{\bar\eta}=X\times_S\bar\eta$ (\ref{fiberschdiv}). We put $Y'_{\bar\eta}=Y'\times_{S'}\bar\eta$ and, for each $i\in I$, put $\bar\eta_i=Q'_i\times_{S'}\bar\eta$. Notice that $Y'_{\bar\eta}$ is a smooth $\bar\eta$-curve and $\bar\eta_i$ $(i\in I)$ are distinct closed points on $Y'_{\bar\eta}$.
By \cite[2.1.1]{lau} (cf. theorem \ref{themdelignelaumon}), we have
\begin{equation}\label{generalineq1}
\sum_{i\in I}\mathrm{dimtot}_{\bar\eta_i}(j_!\mathscr F|_{Y'_{\bar\eta}})\geq\mathrm{dimtot}_z (j_!\mathscr F|_{C}).
\end{equation}
 For every $i\in I$, we have (proposition \ref{Results-1})
\begin{equation}\label{gmiimply}
(\mathrm{DT}_{X_{\bar\eta}}(j_!\mathscr F|_{X_{\bar\eta}}), Y'_{\bar\eta})_{\bar\eta_i}\geq \mathrm{dimtot}_{\bar\eta_i}(j_!\mathscr F|_{Y'_{\bar\eta}}).
\end{equation}
Notice that, for each $i\in I$, the curve $Y'_{\bar\eta}$ intersects the divisor $D_{i,\bar\eta}$ transversally in $X_{\bar\eta}$ at the point $\bar\eta_i$.  Then $(\mathrm{DT}_{X_{\bar\eta}}(j_!\mathscr F|_{X_{\bar\eta}}), Y'_{\bar\eta})_{\bar\eta_i}$ is the coefficient of $D_{i,\eta}$ in the divisor  $\mathrm{DT}_{X_{\eta}}(j_!\mathscr F|_{X_{\eta}})$. It is also the coefficient of $D_i$ in the cycle $\mathrm{GDT}_f(j_!\mathscr F)$. By \eqref{gmiimply}, we have
\begin{equation}\label{generalineq2}
\mathrm{GDT}_f(j_!\mathscr F)\geq \sum_{i\in I}\mathrm{dimtot}_{\bar\eta_i}(j_!\mathscr F|_{Y'_{\bar\eta}})\cdot D_i.
\end{equation}
Since the immersion $\iota:C\rightarrow X_s$ is $SS(j_!\mathscr F|_{X_s})$-transversal at $z$, we have  \eqref{pullbacktocurveundernonchar}
\begin{equation*}
(\mathrm{DT}_{X_s}(j_!\mathscr F|_{X_s}), C)_z=\mathrm{dimtot}_z (j_!\mathscr F|_{C}),
\end{equation*}
which implies that
\begin{equation}\label{generalineq3}
\mathrm{DT}_{X_s}(j_!\mathscr F|_{X_s})=\mathrm{dimtot}_z (j_!\mathscr F|_{C})\cdot (D_s)_{\mathrm{red}},
\end{equation}
since $C$ intersects $(D_s)_{\mathrm{red}}$ transversally  in $X_s$ at $z$. By \eqref{generalineq1}, \eqref{generalineq2} and \eqref{generalineq3}, we have
\begin{eqnarray*}
\rho_{s}^*(\mathrm{GDT}_f(j_!\mathscr F))&\geq&\sum_{i\in I} \mathrm{dimtot}_{\bar\eta_i}(j_!\mathscr F|_{\bar\eta_i})\cdot D_{i,s}\\
&\geq&\mathrm{dimtot}_z (j_!\mathscr F|_{C})\cdot (D_s)_{\mathrm{red}} = \mathrm{DT}_{X_s}(j_!\mathscr F|_{X_s}).\\
\end{eqnarray*}
\end{proof}

Using a similar method, we prove the following proposition which implies that the total dimension number on the generic fiber reaches minimum  in an open dense subset of $S$. It is the key step of proving proposition \ref{geineqprop2}.

\begin{proposition}\label{thge}
Let $\eta$ be the generic point of $S$. We assume that, for each point $t\in S$, the fiber $D_t$ is geometrically irreducible. Then, there exists an open dense subset $V\subseteq S$, such that, for any point $t\in V$, we have
\begin{equation*}
\mathrm{dimtot}_{D_{\eta}}(j_!\mathscr F|_{X_{\eta}})\leq \mathrm{dimtot}_{D_{t}}(j_!\mathscr F|_{X_{t}}).
\end{equation*}

\end{proposition}
\begin{proof}
This is a Zariski local problem at the generic point of $S$. We may assume that $S$ is smooth and that $X$ is irreducible. We put $n=\dim_k X-\dim_k S$. Let $T$ be a connected and smooth $k$-scheme and $\beta:T\ra S$ a flat and generically finite morphism. By \eqref{basechangetd}, to verify the proposition, it is enough to verify it after the base change by $\beta:T\rightarrow S$.

Let $\bar\eta\rightarrow S$ be an algebraic geometric point above $\eta$, $\bar z\in D_{\bar\eta}$ a closed point such that $D_{\bar\eta}$ is smooth at $\bar z$ and that the ramification of $\mathscr F|_{U_{\bar\eta}}$ along $D_{\bar\eta}$ is non-degenerate at $\bar z$. Since $S$ can be replaced by a flat and generically finite cover $T$ as above, we may assume that $\bar z$ can be descended to a $k(\eta)$-rational point $z\in D_{\eta}$. Since the function field $k(\eta)$ of $S$ has infinitely many elements, after shrinking $X$, we can find a smooth $k(\eta)$-curve $C$ and a closed immersion $\iota:C\ra X_{\eta}$ such that the curve $C$ intersects $D_{\eta}$ transversally at $z$ and that the base change $\iota_{\bar\eta}:C_{\bar\eta}\rightarrow X_{\bar\eta}$ of $\iota:C\ra X_{\eta}$ is $SS(j_!\mathscr F|_{X_{\bar\eta}})$-transversal at $\bar z$ (proposition \ref{transversalcurve}).

Choose a regular system of parameters $ t_1,t_2,\cdots, t_n$ of $\mathscr O_{X_{\eta},z}$ such that $(t_1,t_2,\cdots, t_{n-1})$ is the kernel of $\mathscr O_{X_{\eta},z}\rightarrow \mathscr O_{C,z}$
and that $( t_n)$ is the kernel of $ \mathscr O_{X,z}\rightarrow \mathscr O_{D,z}$. We define a $k(\eta)$-morphism $g_{\eta}:k(\eta)[T_1,T_2,\cdots T_{n-1}]\rightarrow \mathscr O_{X_{\eta},z}$ by
\begin{equation*}
g_{\eta}:k(\eta)[T_1,T_2,\cdots T_{n-1}]\rightarrow \mathscr O_{X,z},\ \ \  T_i\mapsto t_i.
\end{equation*}
After shrinking $X$ by a Zariski neighborhood of $z$ again, the map $g_{\eta}$ induces an $S$-morphism $g:X\rightarrow \mathbb A^{n-1}_S$ which satisfies the following conditions
\begin{itemize}
\item[(i)]
it is smooth and of relative dimension $1$;
\item[(ii)]
the restriction $g|_D:D\rightarrow \mathbb A^{n-1}_S$ is \'etale;
\item[(iii)]
the curve $C$ is the pre-image $g^{-1}(\eta\times O)$, where $\eta\times O$ denotes fiber product of $\eta\in S$ and the origin $O\in\mathbb A^{n-1}_k$.
\end{itemize}

Let $\sigma:S\rightarrow \mathbb A^{n-1}_S$ be the zero-section of the canonical projection $\pi:\mathbb A^{n-1}_S\rightarrow S$. We denote by $h:Y\rightarrow S$ the base change of $g:X\rightarrow \mathbb A^{n-1}_S$ by $\sigma:S\rightarrow \mathbb A^{n-1}_S$ and we put $E=Y\times_XD$. Since $h|_E:E\rightarrow S$ is \'etale, there exists a connected \'etale neighborhood $\gamma:S'\ra S$ of the geometric point $\bar\eta\rightarrow S$ such that
\begin{itemize}
\item[(1)]
the fiber product $E'=E\times_S S'$ is a disjoint union of $Q'$ and $P'$;
\item[(2)]
the canonical map $Q'\rightarrow S'$ is an isomorphism and $Q'$ contains a pre-image of $z\in E_{\eta}$.
\end{itemize}
We have the following diagram with Cartesian squares
\begin{equation*}
\xymatrix{\relax
Q'\ar[r]&E'\ar[r]\ar@{}|-(0.5){\Box}[rd]\ar[d]&E\ar@{}|-(0.5){\Box}[rd]\ar[d]\ar[r]&D\ar[d]&\\
&Y'\ar[d]_{h'}\ar[r]\ar@{}|-(0.5){\Box}[rd]&Y\ar[d]^{h}\ar[r]\ar@{}|-(0.5){\Box}[rd]&X\ar[d]^{g}&\\
&S'\ar[r]^{\gamma}&S\ar[r]^-(0.5){\sigma}&\mathbb A^{n-1}_S\ar[r]^-(0.5){\pi}&S}
\end{equation*}
We denote by $\eta'$ the generic point of $S'$. For any geometric point $\bar t'\rightarrow S'$, we put $Y'_{\bar t'}=Y'\times_{S'}\bar t'$, put $Q'_{\bar t'}=Q'\times_{S'} \bar t'$, put $X_{\bar t'}=X\times_S\bar t'$ and put $D_{\bar t'}=D\times_S\bar t'$. Notice that $Y'_{\bar t'}$ is a smooth $\bar t'$-curve and that $Q'_{\bar t'}\in Y'_{\bar t'}$ is a closed point.
 By \cite[2.1.1]{lau} (cf. theorem \ref{themdelignelaumon}), there exists an open dense subset $V'\subseteq S'$ such that, for any $t'\in V'$ and an algebraic geometric point $\bar t'\rightarrow V'$ above $t'$, we have
\begin{equation}\label{gefm1}
\mathrm{dimtot}_{ Q'_{\bar t'}}(j_!\mathscr F|_{Y'_{\bar t'}})=\mathrm{dimtot}_{\bar z}(j_!\mathscr F|_{C_{\bar\eta}}).
\end{equation}
Notice that $Y'_{\bar t'}$ intersects $D_{\bar t'}$ transversally  at $Q'_{\bar t'}$ in $X_{\bar t'}$. Put $t=\gamma(t')$.  We have (proposition \ref{Results-1})
\begin{equation}\label{gefm2}
\mathrm{dimtot}_{D_{t}}(j_!\mathscr F|_{X_{t}})=(\mathrm{DT}_{X_{\bar t'}}(j_!\mathscr F|_{X_{\bar t'}}),Y'_{\bar t'})_{Q'_{\bar t'}}\geq\mathrm{dimtot}_{Q'_{\bar t'}}(j_!\mathscr F|_{Y'_{\bar t'}}).
\end{equation}
By \cite[Corollary 3.9]{wr} (cf. \eqref{pullbacktocurveundernonchar}), we have
\begin{equation}\label{gefm3}
\mathrm{dimtot}_{D_{\eta}}(j_!\mathscr F|_{X_{\eta}})=(\mathrm{DT}_{X_{\bar\eta}}(j_!\sF|_{X_{\bar\eta}}), C_{\bar\eta})_{\bar z}=\mathrm{dimtot}_{\bar z}(j_!\mathscr F|_{C_{\bar\eta}}).
\end{equation}
Hence, by \eqref{gefm1}, \eqref{gefm2} and \eqref{gefm3}, we have $\mathrm{dimtot}_{D_{\eta}}(j_!\mathscr F|_{U_{\eta}})\leq\mathrm{dimtot}_{D_{t}}(j_!\mathscr F|_{U_{t}})$ for any point $t$ in the open dense subset $ \gamma(V')\subseteq S$.
\end{proof}

\begin{proposition}\label{geineqprop2}
There exists an open dense subset $V\subseteq S$ such that, for any point $t\in V$, we have
\begin{equation*}
\rho_t^*(\mathrm{GDT}_f(j_!\mathscr F))\leq \mathrm{DT}_{X_t}(j_!\mathscr F|_{X_t}).
\end{equation*}
\end{proposition}
\begin{proof}
This is a Zariski local problem at the generic point of $S$. We may assume that, for each point $s\in S$ and any indices $i,r\in I$ ($i\neq r$), the fibers $D_{i,s}$ and $D_{r,s}$ have distinct irreducible components (proposition \ref{genetosp2}). Hence the proposition can be reduced to the case where $D$ is irreducible. We may assume that $S$ is integral. By proposition \ref{genetosp1}, there exists a connected smooth $k$-scheme $S'$ and an \'etale map $\gamma:S'\rightarrow S$ such that $D'=D\times_SS'$ is the disjoint union of its irreducible components and that every irreducible component of $D'$ has geometrically irreducible fibers at each point of $S'$.

We put $X'=X\times_SS'$ and denote by $f':X'\rightarrow S'$ the canonical projection. For every point $t'\in S'$, we put $X_{t'}=X\times_St'$. We denote by $\rho'_{t'}:X_{t'}\rightarrow X'$ the canonical injection, by $\gamma':X'\rightarrow X$ the base change of $\gamma:S'\rightarrow S$. We put $t=\gamma(t')$ and denote by $\gamma'_{t}:X_{t'}\rightarrow X_{t}$ the pull-back of $t'\rightarrow t$. We have the following commutative diagram
\begin{equation*}
\xymatrix{\relax
X_{t'}\ar[r]^{\gamma'_{t}}\ar[d]_{\rho'_{t'}}&X_{t}\ar[d]^{\rho_{t}}\\
X'\ar[r]_{\gamma'}&X}
\end{equation*}

There exists an open dense subset $V'\subseteq S'$ such that, for any point $t'\in V'$, we have (proposition \ref{thge})
\begin{equation}\label{generalgeineq}
\rho'^*_{t'}(\mathrm{GDT}_{f'}(j_!\mathscr F|_{X'}))\leq\mathrm{DT}_{X_{t'}}(j_!\mathscr F|_{X_{t'}}).
\end{equation}
By \eqref{basechangetd} and \eqref{finitepbgdt}, we have
\begin{eqnarray*}
\gamma'^*(\mathrm{GDT}_f(j_!\mathscr F))&=&\mathrm{GDT}_{f'}(j_!\mathscr F|_{X'}),\\
\gamma'^*_{t}(\mathrm{DT}_{X_{t}}(j_!\mathscr F|_{X_{t}}))&=&\mathrm{DT}_{X_{t'}}(j_!\mathscr F|_{X_{t'}}).
\end{eqnarray*}
Hence, \eqref{generalgeineq} implies that, for any $t'\in V'$,
\begin{equation}\label{ineqcomdiag}
\gamma'^*_t(\rho^*_t(\mathrm{GDT}_f(j_!\mathscr F)))=\rho'^*_{t'}(\gamma'^*(\mathrm{GDT}_f(j_!\mathscr F)))\leq \gamma'^*_{t}(\mathrm{DT}_{X_{t}}(j_!\mathscr F|_{X_{t}})).
\end{equation}
Since $\gamma'_{t}:X_{t'}\rightarrow X_{t}$ is surjective and \'etale, we deduce from \eqref{ineqcomdiag} that
\begin{equation*}
\rho^*_t(\mathrm{GDT}_f(\mathscr F))\leq \mathrm{DT}_{X_{t}}(j_!\mathscr F|_{X_{t}})
\end{equation*}
for each point $t\in \gamma(V')\subseteq S$.
\end{proof}

\subsection{}\label{closureoft}
Fix a non-closed point of $t\in S$. We denote by $T$ the smooth part of $\overline{\{t\}}$, which is an open dense subset of $\overline{\{t\}}$. We put $X_T=X\times_ST$ and denote by $f_T:X_T\rightarrow T$ the canonical projection, by $\rho_T:X_T\rightarrow X$ the canonical injection and, for any point $s\in T$, by $\iota_s:X_s\rightarrow X_T$ the base change of the inclusion $s\rightarrow T$. We have the following commutative diagram
\begin{equation*}
\xymatrix{\relax
X_s\ar[r]^-(0.5){\iota_s}\ar[rd]_-(0.5){\rho_s}&X_T\ar[d]^{\rho_T}\\
&X}
\end{equation*}

\subsection{}\label{final proof}{\it Proof of theorem \ref{thscdiv}.}

Step 1. By proposition \ref{geineqprop2}, there exists an open dense subset $V\subseteq S$ such that, for  each $t\in V$, we have $\rho_t^*(\mathrm{GDT}_f(j_!\mathscr F))\leq \mathrm{DT}_{X_t}(j_!\mathscr F|_{X_t})$. Meanwhile, for any closed point $s\in V$, we have $\rho_s^*(\mathrm{GDT}_f(j_!\mathscr F))\geq\mathrm{DT}_{X_s}(j_!\mathscr F|_{X_s})$ (proposition \ref{geineqprop1}). Hence, for any closed point $s\in V$, we have
\begin{equation}\label{eqclosedpoint}
\rho_s^*(\mathrm{GDT}_f(j_!\mathscr F))=\mathrm{DT}_{X_s}(j_!\mathscr F|_{X_s}).
\end{equation}

Step 2.  Let $t\in V$ be a non-closed point. We take the notation of \ref{closureoft}. By Step 1, for every closed point $s\in T\cap V$, we have
\begin{equation}\label{pbrf1}
\iota^*_s(\rho^*_T(\mathrm{GDT}_f(j_!\mathscr F)))=\rho_s^*(\mathrm{GDT}_f(j_!\mathscr F))=\mathrm{DT}_{X_s}(j_!\mathscr F|_{X_s}).
\end{equation}
Apply Step 1 to $(\mathscr F|_{U_T}, f_T:X_T\rightarrow T)$. There exists an open dense subset $W\subset T\cap V$ such that, for every closed point $s\in W$, we have
\begin{equation}\label{pbrf2}
\iota^*_s(\mathrm{GDT}_{f_T}(j_!\mathscr F|_{X_T}))=\mathrm{DT}_{X_s}(j_!\mathscr F|_{X_s}).
\end{equation}
Hence, for any closed point $s\in W$, we have (\eqref{pbrf1} and \eqref{pbrf2})
\begin{equation*}
\iota^*_s(\rho^*_T(\mathrm{GDT}_f(j_!\mathscr F)))=\iota^*_s(\mathrm{GDT}_{f_T}(j_!\mathscr F|_{X_T})).
\end{equation*}
By corollary \ref{compcycle}, we obtain that $\rho^*_T(\mathrm{GDT}_f(j_!\mathscr F))=\mathrm{GDT}_{f_T}(j_!\mathscr F|_{X_T})$. Applying $\iota^*_t$ to both sides of the equation, we get
\begin{equation*}
\rho_t^*(\mathrm{GDT}_f(j_!\mathscr F))=\mathrm{DT}_{X_t}(j_!\mathscr F|_{X_t}).
\end{equation*}
Hence \eqref{eqclosedpoint} is valid for any point of $V$, i.e., property (1) of theorem \ref{thscdiv} is proved.

Step 3. Let $t\in S-V$ be a point. If $t$ is closed, we have (proposition \ref{geineqprop1})
\begin{equation}\label{closedx-v}
\rho_t^*(\mathrm{GDT}_f(j_!\mathscr F))\geq \mathrm{DT}_{X_t}(j_!\mathscr F|_{X_t}).
\end{equation}
Assume that $t$ is not closed. We take the notation of \ref{closureoft}. Applying Step 1 to $(\mathscr F|_{U_T}, f_T:X_T\rightarrow T)$, there exists an open dense subset $W\subseteq T$ such that, for any closed point $s\in W$, we have
\begin{equation}\label{pbrf3}
\iota^*_s(\mathrm{GDT}_{f_T}(j_!\mathscr F|_{X_T}))=\mathrm{DT}_{X_s}(j_!\mathscr F|_{X_s}).
\end{equation}
By proposition \ref{geineqprop1}, for each $s\in W$, we have
\begin{equation}\label{pbrf4}
\iota^*_s(\rho^*_T(\mathrm{GDT}_f(j_!\mathscr F)))\geq\mathrm{DT}_{X_s}(j_!\mathscr F|_{X_s}).
\end{equation}
  By \eqref{pbrf3}, \eqref{pbrf4}  and corollary \ref{compcycle}, we get
\begin{equation*}
\rho^*_T(\mathrm{GDT}_f(j_!\mathscr F))\geq \mathrm{GDT}_{f_T}(j_!\mathscr F|_{X_T}).
\end{equation*}
Applying $\iota^*_t$ to both sides of the equation, we obtain again
\begin{equation*}
\rho_t^*(\mathrm{GDT}_f(j_!\mathscr F))\geq\mathrm{DT}_{X_t}(j_!\mathscr F|_{X_t}).
\end{equation*}
Thus, property (2) of theorem \ref{thscdiv} is proved.
\hfill$\Box$

\end{document}